\documentclass[11pt]{article}

\usepackage{latexsym}
\usepackage{amssymb}
\usepackage{amsthm}
\usepackage{amscd}
\usepackage{amsmath}
\usepackage{mathrsfs}
\usepackage[all]{xy}
\usepackage{hyperref} 
\usepackage[usenames,dvipsnames]{color}
\usepackage{graphicx,eepic}

\usepackage{color}

\theoremstyle{definition}
\newtheorem* {theorem*}{Theorem}

\theoremstyle{definition}
\newtheorem{theorem}{Theorem}[section]

\theoremstyle{definition}

\theoremstyle{definition}
\newtheorem{observation}{Observation}[section]

\theoremstyle{definition}

\newtheorem{lemma}{Lemma}[section]
\theoremstyle{definition}

\theoremstyle{definition}

\theoremstyle{definition}

\newtheorem{conjecture}{Conjecture}[section]
\newtheorem{proposition}{Proposition}[section]
\newtheorem{corollary}{Corollary}[section]
\newtheorem* {remark}{Remark}
\newtheorem{example}{Example}[section]
\theoremstyle{definition}

\theoremstyle{definition}
\newtheorem* {remarks}{Remarks}

\xyoption{dvips}

\usepackage{fullpage}

\numberwithin{equation}{section}

\def\({\left(}
\def\){\right)}
   \newcommand{\FF}{\mathbb{F}}  \newcommand{\CC}{\mathbb{C}}  \newcommand{\QQ}{\mathbb{Q}}    
  \newcommand{\cO}{\mathcal{O}} 
\newcommand{\cR}{\mathcal{R}}

        \def\Ind{\mathrm{Ind}} \def\GL{\mathrm{GL}}   \def\Res{\mathrm{Res}}       
\def\Irr{\mathrm{Irr}}  \def\wt{\widetilde}

   \newcommand{\fkn}{\mathfrak{n}}
\newcommand{\h}{\mathfrak{h}}

\newcommand{\one}{{1\hspace{-.11cm} 1}}

\def\fk{\mathfrak}

\def\barr{\begin{array}}
\def\earr{\end{array}}
\def\ba{\begin{aligned}}
\def\ea{\end{aligned}}
\def\be{\begin{equation}}
\def\ee{\end{equation}}

\def\ol{\widehat}
\def\olfkl{\ol{\fk l}}
\def\olfks{\ol{\fk s}}
\def\oll{\ol{L}}
\def\ols{\ol{S}}

\def\UT{\mathrm{UT}}
\def\fkt{\fk{u}}
\def\logpsi{\psi^{\exp}}
\def\tobedecided{}
\def\cF{\mathrm{F}}
\def\exp{\mathrm{Exp}}

\makeatletter
\renewcommand{\@makefnmark}{\mbox{\textsuperscript{}}}
\makeatother

\allowdisplaybreaks[1]

\UseCrayolaColors

\begin{document}
\title{Iterative character constructions for algebra groups}
\author{Eric Marberg\footnote{This research was conducted with government support under
the Department of Defense, Air Force Office of Scientific Research, National Defense Science
and Engineering Graduate (NDSEG) Fellowship, 32 CFR 168a.} \\ Department of Mathematics \\ Massachusetts Institute of Technology \\ \tt{emarberg@math.mit.edu}}
\date{}

\maketitle

\setcounter{tocdepth}{2}

\begin{abstract}
We construct a family of orthogonal characters  of an algebra group which decompose the supercharacters defined by Diaconis and Isaacs \cite{DI}.  
Like supercharacters, these characters are given by nonnegative integer linear combinations of Kirillov functions and are induced from linear supercharacters of certain algebra subgroups.  We derive a formula for these characters and give a condition for their irreducibility; generalizing a theorem of Otto \cite{O}, we also show that each such character has the same number of Kirillov functions and irreducible characters as constituents.  In proving these results, we observe as an application how a recent computation by Evseev \cite{E} implies that every irreducible character of the unitriangular group $\UT_n(q)$ of unipotent $n\times n$ upper triangular matrices over a finite field with $q$ elements is a Kirillov function if and only if $n\leq 12$.  As a further application, we discuss some more general conditions showing that Kirillov functions are  characters, and describe some results related to counting the irreducible constituents of supercharacters.
\end{abstract}

\section{Introduction}

Let $\fkn$ be a nilpotent, finite-dimensional, associative algebra over a finite field $\FF_q$ and let $G=1+\fkn$ denote the \emph{algebra group} of formal sums $1+X$ for $X \in \fkn$.  
Algebra groups comprise a relatively well-behaved class of $p$-groups, including as prototypical examples the Sylow $p$-subgroups of the general linear group $\GL(n,\FF_{q})$, where $p$ is the characteristic of $\FF_q$.  Their representations have several notable properties; in particular, by theorems of Isaacs \cite{I95} and Halasi \cite{H} each irreducible representation of an algebra group has $q$-power degree  and is induced from a linear representation of an algebra subgroup, by which we mean an algebra group attached to a subalgebra of $\fkn$.  If the characteristic of $\FF_q$ exceeds the nilpotency class of $\fkn$, then a version of Kirillov's orbit method for finite groups succeeds in constructing all irreducible representations of $G$ (see \cite{Sangroniz}).
When the characteristic of $\FF_q$ is small, however, a general description of the irreducible representations of $G$ remains unknown. 

Addressing this problem, Diaconis and Isaacs introduced in \cite{DI} a construction  which gives a useful approximation to the irreducible representations of an arbitrary algebra group. 
The character version of this goes as follows.  
Write $\fkn^*$ for the dual space of $\FF_q$-linear maps $\fkn \to \FF_q$ and fix  a nontrivial homomorphism $\theta : \FF_q^+\to \CC^\times$ so that the irreducible characters of the additive group $\fkn$ are precisely the maps $\theta \circ \lambda$ for $\lambda \in \fkn^*$.  
%
%
If one defines $\fk l_\lambda \subset \fkn$ as the left kernel of the bilinear form 
\[  \barr{cccc} B_\lambda : & \fkn \times \fkn & \to & \FF_q \\ & (X,Y) & \mapsto & \lambda(XY)\earr\] then $\fk l_\lambda$ is a subalgebra, and the function $\theta_\lambda(1+X) \overset{\mathrm{def}}= \theta\circ \lambda(X)$ defines a linear character of  the corresponding algebra subgroup $L_\lambda \overset{\mathrm{def}} = 1 + \fk l_\lambda$.
Diaconis and Isaacs define the \emph{supercharacter} indexed by $\lambda \in \fkn^*$ as the induced character $\chi_\lambda\overset{\mathrm{def}}=\Ind_{L_\lambda}^G(\theta_\lambda)$.  

While the characters $\chi_\lambda$ are often reducible, their constituents partition the set of all irreducible characters of $G$, and they are constant on certain unions of conjugacy classes.  These and other properties make them an often practical substitute for the typically unknown irreducible characters of $G$.  Nevertheless, the supercharacters of $G$ come from a coarse enough construction that they often fail to shed much light on interesting properties of the group's irreducible representations.

This work concerns an extension of Diaconis and Isaacs's construction which produces a typically  much larger family of orthogonal characters of an algebra group.  
%
%
%
%
Our extension derives from the following observation.   Given $\lambda \in \fkn^*$, define $\fk s_\lambda$ as the left kernel of the restriction of $B_\lambda$ from $\fkn\times \fkn$ to the domain $\fkn \times \fk l_\lambda$.  Then $\fk s_\lambda$ is a subalgebra, and induction from the algebra subgroup $S_\lambda \overset{\mathrm{def}}=1+\fk s_\lambda$ to $G$ defines a bijection 
\be\label{1} \biggl\{ \text{Irreducible constituents of $\Ind_{L_\lambda}^{S_\lambda}(\theta_\lambda)$}\biggr\} \to \biggl\{\text{Irreducible constituents of $\chi_\lambda = \Ind_{L_\lambda}^G(\theta_\lambda)$}\biggr\}.\ee
Andr\'e and Nicol\'as first proved this fact in an equivalent but slightly less elementary form in \cite{AndreAdjoint}.  
We wish to replace the supercharacter $\chi_\lambda$ with a distinguished constituent to be called $\xi_\lambda$.  
At present we have no way in general to describe the irreducible constituents of the induced character $\Ind_{L_\lambda}^{S_\lambda}(\theta_\lambda)$.  However, $\Ind_{L_\lambda}^{S_\lambda}(\theta_\lambda)$ does have one obvious constituent:  namely, the supercharacter of $S_\lambda$ indexed by the restriction $\mu \in (\fk s_\lambda)^*$ of the linear map $\lambda: \fkn\to \FF_q$.   Thus, given an algebra group $G$ with a supercharacter $\chi_\lambda$, one can produce an algebra subgroup $S_\lambda \subset G$ with a supercharacter $\chi_\mu$, such that $\Ind_{S_\lambda}^G(\chi_\mu)$ is a constituent of $\chi_\lambda$.  
By applying this ability inductively, one obtains a descending chain of algebra subgroups $G\supset S_\lambda \supset \dots$, and we define $\xi_\lambda$ as the character  of $G$ induced from the supercharacter attached to the chain's terminal element.   

  In Section \ref{3} below we repeat this construction in greater detail and prove a number of useful properties concerning the characters $\xi_\lambda$.  For example, we show that $\xi_\lambda$ has a formula given by a sum of functions $\theta_\nu$ over $\nu$ in a union of $k$ coadjoint orbits in $\fkn^*$, where $k$ is the number of the character's irreducible constituents.   Likewise, we describe  when two characters $\xi_\lambda$, $\xi_\nu$ are distinct and how to compute the inner product $\langle \xi_\lambda, \xi_\nu \rangle_G$.  We mention that Andr\'e and Nicol\'as prove a few results concerning the existence of the characters $\{\xi_\lambda\}$ in \cite{AndreAdjoint}; their derivations come mostly from recursive applications of the bijection (\ref{1}), however.  Our contribution beyond this earlier work is two-fold.  First, we  
 provide a more elementary definition of $\xi_\lambda$ in terms of the bilinear form $B_\lambda$, and second, we employ this definition to explicitly compute various properties of the characters $\xi_\lambda$.

\def\Ad{\mathrm{Ad}}

Two other topics pervade and motivate our results: first, the question of how to determine whether a Kirillov function is a character, and second, the problem of how to count the irreducible constituents of supercharacters.  We recall that there is a coadjoint action of the algebra group $G=1+\fkn$ on $\fkn^*$, given by $g : \lambda \mapsto \lambda \circ \Ad(g)^{-1}$ where $\Ad(g)(X) = gXg^{-1}$ for $g \in G$ and $X \in \fkn$.
For each $\lambda\in \fkn^*$, the corresponding \emph{Kirillov function} $\psi_\lambda : G \to \CC$ is the complex-valued function 
\[ \psi_\lambda(g) =|\Omega|^{-1/2} \sum_{\mu \in \Omega} \theta\circ \mu(g-1),\qquad\text{where $\Omega\subset \fkn^*$ is the coadjoint orbit of $\lambda$.}\]  We describe some conditions which imply that  $\psi_\lambda$ is a character, and using a recent computational result due to Evseev \cite{E}, prove this theorem:

\begin{theorem*}
Fix a prime power $q>1$ and let $\UT_n(q)$ denote the group of unipotent $n\times n$ upper triangular matrices over $\FF_q$.  Then the set of irreducible characters of $\UT_n(q)$  is equal to the set of Kirillov functions of $\UT_n(q)$ if and only if $n\leq 12$.\end{theorem*}

We also consider \emph{polynomial Kirillov functions}, by which we mean compositions of Kirillov functions with polynomial bijections $G\to G$ of the form $1+X\mapsto 1+X + \sum_{k\geq 2} a_k X^k$ with $a_k \in \FF_q$.  Such functions are of interest because if the polynomial bijection is chosen as the inverse of an appropriately truncated exponential map, then the resulting \emph{exponential Kirillov functions} give all irreducible characters of $G$ when the characteristic of $\FF_q$ exceeds the nilpotency class of $\fkn$  (see \cite{Sangroniz}). 
 Generalizing a result of Otto \cite{O}, we prove the following:

\begin{theorem*}
A polynomial Kirillov function attached to the coadjoint orbit of some $\lambda \in \fkn^*$ is a complex linear combination of the irreducible constituents of the character $\xi_\lambda$.
\end{theorem*}

All this is useful for  decomposing supercharacters into smaller constituents, and for investigating the properties of individual irreducible characters of algebra groups.  For example, it is known that for $n\geq 13$, $\UT_n(2)$ has irreducible characters with non-real values (see, for example, \cite{IK1,IK2}).  Existing proofs of this fact have all been non-constructive and have relied in some way on computer calculations.  By using the results in this work, however, one can derive this property by hand.  
We carry out the details of this application
in the supplementary work \cite{supp1}, where we explicitly construct irreducible characters of $\UT_n(q)$ with values in arbitrarily large cyclotomic fields.  A second application of the results herein appears in the paper \cite{supp2}, in which we consider the problem of counting the irreducible characters of $\UT_n(q)$ with a fixed degree and derive polynomial expressions for the number of irreducible characters $\UT_n(q)$ of degree $\leq q^8$.


\section{Preliminaries}\label{prelim-sect}

Here we briefly  establish our notational conventions, then define algebra groups, Kirillov functions, and supercharacters, and identify their noteworthy properties.

\subsection{Conventions and notation}

Throughout, all groups $G$ are finite and all algebras $\fkn$ are finite-dimensional and associative.  By a representation of $G$, we shall mean a homomorphism from $G$ into the automorphism group  of a finite-dimensional vector space over a subfield of the complex numbers $\CC$.  A character is any function $G \to \CC$ afforded by the trace of a representation. 
We let $\langle\cdot,\cdot\rangle_G$ denote the standard inner product on the complex vector space of functions $G \to \CC$ defined by 
$ \langle f,g\rangle_G = \frac{1}{|G|} \sum_{x \in G} f(x) \overline{g(x)}$, 
and write $\Irr(G)$ to denote the set of  irreducible characters of $G$, or equivalently the set of characters $\chi$ of $G$ with $\langle \chi,\chi \rangle_G = 1$.  A function $ G\to \CC$ is then a character if and only if it is a nonzero sum of irreducible characters with nonnegative integer coefficients.

  If $\psi$ and $\chi$ are characters of $G$, then we say that $\psi$ is a constituent of $\chi$ if $\chi-\psi$ is a character of $G$ or zero.  
  In this case, the multiplicity of $\psi$ in $\chi$ is defined as the largest integer $m$ such that $\chi-m\psi$ is a character or zero. 
For any character $\chi$ of $G$, we let 
 \[\Irr(G,\chi) = \{ \psi \in \Irr(G) : \langle \chi, \psi \rangle_G > 0\}\]  denote the set of irreducible constituents of $\chi$.

 If $f: S \to T$ is a map and $S'\subset S$, then we write $f\downarrow S'$ to denote the restricted map $S'\to T$.  For functions on groups, we may also write $\Res_H^G(\chi) = \chi\downarrow H$ to denote the restriction of $\chi : G\to \CC$ to a subgroup $H$.  
If $\chi$ is any complex valued function whose domain includes the subgroup $H\subset G$, then we define the induced function $\Ind_{H}^G(\chi):G \to \CC$ by the formula
 \be\label{frob} \Ind_H^G(\chi)(g) = \frac{1}{|H|} \sum_{\substack { x \in G  \\ xgx^{-1} \in H}} \chi(xgx^{-1}),\qquad\text{for }g\in G.\ee  We recall that restriction takes characters of $G$ to characters of $H$ and induction takes characters of $H$ to characters of $G$.

Throughout we write $\fkt_n(q)$ and $\UT_n(q)$ to refer to the algebra and algebra group 
\[ \ba \fkt_n(q) & = \text{the set of upper triangular $n\times n$ matrices over $\FF_q$ with zeros on the diagonal,} \\
\UT_n(q)&=\text{the set of upper triangular $n\times n$ matrices over $\FF_q$ with ones on the diagonal.}\ea\]  Given integers $1\leq i<j \leq n$ we let 
\[ \ba e_{ij} & = \text{the matrix in $\fkt_n(q)$ with 1 in position $(i,j)$ and zeros elsewhere,}\\
e_{ij}^* &=\text{the $\FF_q$-linear map $\fkt_n(q)\to \FF_q$ given by $e_{ij}^*(X) = X_{ij}$.}\ea\]
These matrices and maps are then dual bases of $\fkt_n(q)$ and its dual space $\fkt_n(q)^*$.


\subsection{Algebra groups and their actions}

Let $\fkn$ be a (finite-dimensional, associative) nilpotent $\FF_q$-algebra, and $\fkn^*$ its dual space of $\FF_q$-linear maps $\fkn \to \FF_q$.  Write $G = 1+\fkn$ to denote the corresponding \emph{algebra group}; this is the set of formal sums $1+X $ with $X \in \fkn$, made a group via the multiplication \[(1+X)(1+Y) = 1+X+Y+XY.\]  As ubiquitous examples, we take $\fkn$ to be the algebra $\fkt_n(q)$ of strictly upper triangular $n\times n$ matrices over $\FF_q$ and $G$ to be the \emph{unitriangular group} $\UT_n(q) = 1 + \fkt_n(q)$.  

We call a subgroup of $G=1+\fkn$ of the form $H = 1+\h$ where $\h \subset \fkn$ is a subalgebra an \emph{algebra subgroup}.  
If $\h \subset \fkn$ is a two-sided ideal then $H$ is a normal algebra subgroup of $G$, and the map $gH \mapsto 1+(X+\h)$ for  $g=1+X \in G$ gives an isomorphism $G/H \cong 1 + \fkn /\h$.  In practice we shall usually identify the quotient $G/H$ with the algebra group $1 + \fkn /\h$ by way of this canonical map.


\def\osdef{\overset{\mathrm{def}}}

We have several actions of $G$ on its dual space $\fkn^*$.  The 
 \emph{coadjoint} action of $G$ on $\fkn^*$ is the right action given by 
 $(\lambda,g) \mapsto \lambda^{g}$ where we define
\[ 
\lambda^g(X) = \lambda(g X g^{-1}),\qquad\text{for } \lambda \in \fkn^*,\ g\in G,\ X \in \fkn.\]  We denote the 
 coadjoint orbit of $\lambda \in \fkn^*$ by 
 $\lambda ^G$.  
The size of this orbit is $|\lambda^G|=q^e$ where $e$ is an even integer, equal to the codimension in $\fkn$ of the 
 radical of the alternating $\FF_q$-bilinear form $\fkn \times \fkn \to \FF_q$ given by $(X,Y) \mapsto  \lambda(XY-YX)$ \cite[Lemma 4.4]{DI}.

We also have left and right actions of $G$ on $\fkn^*$.
The group $G$ acts on the left and right on $\fkn$  by multiplication, and on 
 $\fkn^*$ by $(g , \lambda) \mapsto g\lambda$ and $( \lambda,g) \mapsto \lambda g$ where we define
 \[ g\lambda(X) = \lambda(g^{-1}X)\qquad\text{and}\qquad \lambda g(X) = \lambda(Xg^{-1}),\qquad\text{for }\lambda \in \fkn^*,\ g \in G,\ X \in \fkn.\]
 These actions commute, in the sense that 
 $(g\lambda) h = g(\lambda h)$ for $g,h \in G$, so there is no ambiguity in removing all parentheses and writing expressions like 
 $g\lambda h$.   We denote the left, right, and two-sided orbits of 
 $\lambda \in \fkn^*$ by 
 $G\lambda$, $\lambda G$, $G\lambda G$.   We have 
 $ |G\lambda | = |\lambda G|$ 
 and  $|G\lambda G| = \frac{|G\lambda| |\lambda G|}{|G\lambda \cap \lambda G|}$ for $\lambda \in \fkn^*$, 
 and these are all nonnegative integer powers of $q$ by \cite[Lemmas 3.1 and 4.2]{DI}.

\subsection{Functions on algebra groups}\label{alg}
 
 For the duration of this work, $\theta : \FF_q^+\to \CC^\times$ denotes a fixed, nontrivial homomorphism from the additive group of $\FF_q$ to the multiplicative group of nonzero complex numbers.  Observe that $\theta$ takes values in the cyclotomic field $\QQ(\zeta_p)$, where $p>0$ is the characteristic of $\FF_q$ and $\zeta_p = e^{2\pi i / p}$ is a primitive $p$th root of unity.

 The dual space $\fkn^*$ naturally indexes several different kinds of complex-valued functions on $G$.  To begin,
for each $\lambda \in \fkn^*$, we define $\theta_\lambda : G \to \QQ(\zeta_p)$ as the function with
\[ \theta_\lambda(g) = \theta\circ \lambda(g-1),\qquad\text{for }g \in G.\]  
The maps $\theta\circ \lambda : \fkn \to \CC$ are the distinct irreducible characters of the abelian group $\fkn$, and from this it follows that the functions $\theta_\lambda : G\to \CC$ are an orthonormal basis (with respect to $\langle\cdot,\cdot\rangle_G$) for all functions on the group.   

Next, for $\lambda \in \fkn^*$ the \emph{Kirillov function} $\psi_\lambda$ and the \emph{supercharacter} $\chi_\lambda$ are the functions  $G \to \QQ(\zeta_p)$ defined by 
\be\label{superchar-def} \psi_\lambda = \frac{1}{\sqrt{|\lambda^G|}} \sum_{\mu \in \lambda^G} \theta_\mu\qquad\text{and}\qquad 
 \chi_\lambda = \frac{|G\lambda|}{|G\lambda G|} \sum_{\mu \in G \lambda G} \theta_\mu.
\ee 
We note from the remarks in the previous section that the degrees $\psi_\lambda(1) = \sqrt{|\lambda^G|}$ and $\chi_\lambda(1) = |G\lambda| = |\lambda G|$  of these functions are nonnegative integer powers of $q$. 
The choice of $\theta$ in these definitions is not important, as a different choice yields the same set of functions but with the indices $\lambda$ permuted.  Kirillov functions are sometimes but not always characters, while supercharacters are always characters but often reducible.  The following remarks detail some more important properties of these functions.

\begin{remarks}
\begin{enumerate}
\item[]


\item[(i)]  We have $\psi_\lambda = \psi_\mu$ if and only if $\lambda^G = \mu^G$ and $\chi_\lambda = \chi_\mu$ if and only if $G\lambda G = G\mu G$.  In fact, the number of distinct Kirillov functions is equal to the number of conjugacy classes in $G$ and the number of distinct supercharacters is equal to the number of two-sided $G$-orbits in $\fkn$ \cite[Lemma 4.1]{DI}.
The orthogonality of the functions $\theta_\mu$ implies that
\[\langle \psi_\lambda,\psi_\mu\rangle_G = \left\{\barr{ll} 1, &\text{if }\lambda^G = \mu^G, \\
0,&\text{otherwise,}\earr\right.\quad\text{and}\quad 
\langle \chi_\lambda,\chi_\mu\rangle_G = \left\{\barr{ll} |G\lambda \cap \lambda G|, &\text{if }G \lambda G = G \mu G, \\
0,&\text{otherwise.}\earr\right.\]  Kirillov functions are constant on conjugacy classes, and thus form an orthonormal basis for the class functions on $G$.
Also, $\chi_\lambda$ is a linear character if and only if $\chi_\lambda =\psi_\lambda= \theta_\lambda$ and $\langle \chi_\lambda, \chi_\lambda \rangle_G=1$ if and only if $\chi_\lambda = \psi_\lambda$ \cite[Theorem 5.10]{DI}.

\item[(ii)] The column orthogonality relations satisfied by the characters of the abelian group $\fkn$ imply that $\sum_{\nu \in \fkn^*} \theta\circ \nu(X)=0$ if $X \in \fkn \setminus \{0\}$.  This observation shows that the character $\rho_G$ of the regular representation of $G$ decomposes as 
\be\label{decomp} 
\rho_G = \sum_{\nu \in \fkn^*} \theta_\nu = \sum_{\mu} \psi_\mu(1) \psi_\mu
=
 \sum_{\lambda} \frac{|G\lambda G|}{|G\lambda|} \chi_\lambda
 \ee 
 where the second sum is over a set of representatives $\mu$ of the coadjoint orbits in $\fkn^*$ and the third sum is over a set of representatives $\lambda$ of the two-sided $G$-orbits in $\fkn^*$.    Thus every irreducible character of $G$ appears as a constituent of a unique supercharacter.

\item[(iii)] If $\h \subset \fkn$ is a subalgebra and $H = 1+\h \subset G$ is an algebra subgroup,  then $\Res_H^G(\psi_\lambda)$ (respectively, $\Res_H^G(\chi_\lambda)$) is a nonnegative integer linear combination of Kirillov functions  (respectively, supercharacters), by \cite[Theorem 5.5]{AndreAdjoint} and 
 \cite[Theorem 6.4]{DI}.  Since Kirillov functions are orthonormal, it follows that if $\mu \in \h^*$ then $\Ind_H^G(\psi_\mu)$ is a nonnegative integer linear combination of Kirillov functions.  Induction may fail to take sums of supercharacters to sums of supercharacters, and so to make an analogous statement one must use a different notion of induction to go from functions on $H$ to functions on $G$; see \cite{MT} for a more detailed discussion.

\item[(iv)] If $\lambda \in \fkn^*$ and $\vartheta_1,\dots,\vartheta_m \in \Irr(G)$ are the distinct irreducible constituents of $\chi_\lambda$, then there are complex numbers $c_i$ such that $\psi_\lambda = \sum_{i=1}^m c_i \vartheta_i$ \cite[Theorem 2.5]{O}.  Since Kirillov functions and irreducible characters are both orthonormal bases of the class functions on $G$, it follows that the number of irreducible constituents of $\chi_\lambda$ is the number of coadjoint orbits in $G\lambda G$.  
 \end{enumerate}
\end{remarks}

\subsection{Supercharacters}\label{superchars}

Andr\'e \cite{Andre1} first defined what we call supercharacters for the algebra group $\UT_n(q)$; in this special case, they provide as an accessible substitute for the group's unknown and presumably unclassifiable irreducible representations.  Several years later, Yan \cite{Yan} showed how one could replace Andr\'e's definition with a more elementary construction, which Diaconis and Isaacs \cite{DI} subsequently generalized to algebra groups. 

 There are several ways to see that the functions $\chi_\lambda$ are characters. Here is the interpretation of greatest relevance to this work: 
if $L_\lambda = \{ g \in G : g\lambda = \lambda\}$ denotes the left stabilizer of $\lambda$ in $G$, then $L_\lambda$ is an algebra subgroup of $G$, $\theta_\lambda$ restricts to a linear character of $L_\lambda$, and 
\[\chi_\lambda = \Ind_{L_\lambda}^G(\theta_\lambda)\] by \cite[Theorem 5.4]{DI}.  It is an instructive exercise to derive (\ref{superchar-def}) from this definition, via the Frobenius formula for induced characters.

Supercharacters are constant on \emph{superclasses}, by which we mean sets of the form $\{ 1+ gXh : g,h\in G\}$ for $X \in \fkn$.  The singleton set $\{1\}$ is a superclass, and in general the superclasses are unions of conjugacy classes which partition $G$ and are equal in number to the supercharacters of $G$.   Remark (i) in Section \ref{alg} implies that the  distinct supercharacters $\chi_\lambda$ form an orthogonal basis for the \emph{superclass functions} on $G$, and so the characters $\{\chi_\lambda \in \fkn^*\}$ form a \emph{supercharacter theory} of $G$ in the sense of Diaconis and Isaacs in \cite{DI}.   For $G= \UT_n(q)$, this supercharacter theory has many interesting combinatorial properties.  For example, the superclasses and supercharacters are in bijection with set partitions whose arcs are labeled by elements of $\FF_q^\times$, 
 and the ring of superclass functions has a natural Hopf algebra structure isomorphic to a colored version of the Hopf algebra $NCSym$ of symmetric functions in non-commuting variables; see \cite{alia}.

Following \cite{AndreAdjoint}, we say that the map $\lambda \in \fkn^*$ and the supercharacter $\chi_\lambda$ are \emph{fully ramified} if $G\lambda= \lambda G$.  This is equivalent to the condition that $\chi_\lambda(1) = \langle\chi_\lambda,\chi_\lambda\rangle_G$  or that $\chi_\lambda$ be the character of a two-sided ideal in the group algebra $\CC G$, and in this case $\chi_\lambda =  \sum_{\nu \in G\lambda G} \theta_\nu$.  
 In this situation $L_\lambda \vartriangleleft G$,  $\chi_\lambda$ vanishes on $G-L_\lambda$, and $\theta_\lambda$ restricts to a $G$-invariant character of $L_\lambda$; in fact, by \cite[Theorem 4.5]{AndreAdjoint} these properties are equivalent to the definition. Comparing this observation with the standard use of  ``fully ramified'' for irreducible characters (cf.\ Problem (6.3) in \cite{Isaacs}) explains our appropriation of the term.
 When $G=\UT_n(q)$, there are simple formulas due to Yan \cite{Yan} computing $\chi_\lambda(1)$ and $\langle\chi_\lambda,\chi_\lambda\rangle_G$ for all supercharacters, and these  show that the unitriangular group's fully ramified supercharacters  are precisely its linear characters.
This nice description certainly does not hold in general, and 
we shall see that many problems regarding the decomposition of supercharacters into irreducibles \emph{a priori} reduces to the fully ramified case.

\subsection{Kirillov functions}\label{kir}

Let us now turn to Kirillov functions. The set  $\{ \psi_\lambda:\lambda \in \fkn^*\}$ is an orthonormal basis for the class functions of $G$, and so if every Kirillov function is a character then the set of these functions is precisely $\Irr(G)$.  Kirillov \cite{K} conjectured that this is the case when $G = \UT_n(q)$.   Isaacs and Karagueuzian disproved this conjecture using two computational approaches: first, by writing down a matrix in $\UT_{13}(2)$ not conjugate to its inverse \cite{IK1,IK2}, and second, by elegantly combining the Frobenius-Schur involution counting formula with an algorithm to compute the character degrees of $\UT_n(2)$ \cite{IK05}.  
Both results (non-constructively) prove the existence of a non-real irreducible character of $\UT_n(2)$ for $n>12$. 

We will discuss in Section \ref{inductive} how a recent  computation due to Evseev \cite{E} implies that Kirillov's conjecture holds for $\UT_n(q)$ if and only if $n\leq 12$, regardless of the choice of $q$. 
%
%
Determining when Kirillov's conjecture holds for other algebra groups and more generally when a given Kirillov function $\psi_\lambda$ is a character remains an open problem.  
 Results in Section \ref{3}, however,  in some sense reduce the second problem 
  to the case where $\lambda$ is fully ramified, i.e., when $G\lambda = \lambda G$.

In our definition of a Kirillov function we seek to attach a coadjoint orbit in $\fkn^*$ to an irreducible character of $G$ 
 by way of the bijection $\fkn \to G$ given by $X \mapsto 1+X$.  This is Kirillov's orbit method in the context of finite groups as described in \cite{K}.  In practice, one is more successful in developing a correspondence between coadjoint orbits and irreducible characters if a different bijection $\fkn \to G$ is used.  To formalize this, if $\cF : \fkn \to G$ is any bijection and $\psi : G \to \CC$, we define $\psi^\cF : G\to \CC$ by 
 \be\label{formal} \psi^\cF\( \cF(X)\) = \psi(1+X),\qquad\text{for }X \in \fkn.\ee
 Recalling that $\fkn$ is an algebra over a finite field of characteristic $p>0$,  let $\tobedecided \exp : \fkn \to G$ be the truncated exponential map \[\ba & \barr{l}\tobedecided \exp (X) = 1 + X + \frac{1}{2}X^2 + \frac{1}{6}X^3 + \dots + \frac{1}{(p-1)!}X^{p-1}. 
\earr\ea\]  This map is always a bijection (as is any polynomial map with constant term one and nonzero linear term), and so can we may consider the functions $\logpsi_\lambda : G \to \QQ(\zeta_p)$ for $\lambda \in \fkn^*$ as defined by (\ref{formal}); we call these maps \emph{exponential Kirillov functions}.
In characteristic two $\logpsi_\lambda = \psi_\lambda$, 
 and in all prime characteristics  the following observation holds:

\begin{observation}\label{obs1} Remarks (i)-(iii) in Section \ref{alg} remain valid if we replace all instances of Kirillov functions $\psi_\lambda$ with $\psi^\cF_\lambda$ where $\cF : \fkn \to G$ is any polynomial map of the form 
$ \cF(X) = 1 +X + \sum_{k\geq 2} a_k X^k$ with $a_k \in \FF_q$. 
  In particular, the exponential Kirillov functions form an orthonormal basis for the class functions of $G$.
\end{observation}

Note well that we do not yet claim that Remark (iv) still holds.  It does, and this will be a corollary of Theorem \ref{otto-gen} below, whose proof in Section \ref{inductive} shall require a nontrivial effort.

\def\cG{\mathrm{F'}}

\begin{proof}
The map $\cF$ is a bijection, because if $\cG : G \to \fkn$ has the form $\cG(1+X) = \sum_{k=1}^n b_k X^k$ where $b_k \in \FF_q$ and $\fkn^n =0$,  then the system of equations  in the variables $b_k$ determined by writing $\cG \circ \cF (X) = X$ has a solution; in particular, we must have $b_1 = 1$ and each $b_{k+1}$ is determined by $b_1,\dots,b_k$.   Note that $\cF$ thus defines a bijection $\h \to H= 1+\h$ for any subalgebra $\h \subset \fkn$.  
Next observe that $\cF(gXg^{-1}) = g\cdot\cF(X)\cdot g^{-1}$ for $g\in G$ and $X \in \fkn$, so $\psi^\cF_\lambda$ is a class function.  Given these two facts, everything but (iv) and the last statement in (i)  is straightforward.  To see that (i) holds, we note that $\cF(X)$ and $1+X$ lie in the same superclass 
so $\chi_\lambda = \chi_\lambda^{\cF}$, and thus $\langle \chi_\lambda,\chi_\lambda\rangle_G=1$ if and only if $\chi_\lambda = \psi_\lambda$ if and only if $\chi_\lambda = \chi_\lambda^{\cF}= \psi^\cF_\lambda$.  
\end{proof}

This notation is relevant primarily because of the following theorem, first observed by Kazhdan \cite[Propositions 1 and 2]{Kaz}  for unitriangular matrix groups and proved in greater generality by   Sangroniz \cite[Corollary 3]{Sangroniz}.

\begin{theorem}[Sangroniz \cite{Sangroniz}] \label{Corollary 3}
Let $\fkn$ be a nilpotent $\FF_q$-algebra write $G=1+\fkn$. If $p>0$ is the characteristic of $\FF_q$ and $\fkn^p=0$, then 
$\Irr(G) =\left \{\logpsi_\lambda  : \lambda \in \fkn^*\right\}$.  
\end{theorem}

If $\fkn = \fkt_n(q)$ then $\fkn^p=0$ if and only if $n<p$.  Sangroniz notably improves this bound by a factor of two:
\begin{corollary}[Sangroniz \cite{Sangroniz}]\label{sang} If 
$n<2p$ then $\Irr\(\UT_n(q)\) =\left \{\logpsi_\lambda  : \lambda \in \fkt_n(q)^*\right\}$. 
\end{corollary}

Thus exponential Kirillov functions give us more characters than ordinary Kirillov functions in a quite concrete sense, but only in odd characteristic.  Still, $\logpsi_\lambda$ can fail to be a character because its values all lie in $\QQ(\zeta_p)$ yet for $n>6p$ there are characters of $\UT_n(q)$ which take values outside $\QQ(\zeta_p)$ 
(see \cite{IK1,IK2,IK05,VeraLopez2004}).  We actually construct such characters in 
\cite{supp1}, and the calculations undertaken there suggest the following 
conjecture.


\begin{conjecture}
$\Irr\(\UT_n(q)\) = \left\{ \logpsi_\lambda : \lambda \in \fkt_n(q)^*\right\}$ if and only if $n\leq 6p$.  
\end{conjecture}


\section{Iterative character constructions}\label{3}

Throughout this section, $\fkn$ denotes a fixed nilpotent $\FF_q$-algebra and $G=1+\fkn$.
As described above, each $\lambda \in \fkn^*$ naturally indexes the two functions $G\to \CC$ given by the Kirillov function $\psi_\lambda$ and the supercharacter $\chi_\lambda$.  In this section we associate to $\lambda$ a third map $\xi_\lambda : G\to \CC$, in some sense interpolating between $\psi_\lambda$ and $\chi_\lambda$.
Like $\chi_\lambda$ the function $\xi_\lambda$ will always be a character.  However, $\xi_\lambda$ will typically have fewer irreducible constituents, and so will provide a more accurate, if less computable, approximation to the irreducible characters of $G$.    
  The crucial step in our constructions depends on a result concerning the irreducible constituents of a supercharacter, first observed by Andr\'e for the unitriangular group \cite{AndreHecke} and subsequently generalized to algebra groups by Andr\'e and Nicol\'as \cite{AndreAdjoint}.

\subsection{Decomposing supercharacters}\label{fk}

\def\rad{\mathrm{rad}}
\def\radl{\ker_{\mathrm{L}}}
\def\radr{\ker_{\mathrm{R}}}

For each $\lambda \in \fkn^*$,  define two subspaces $\fk l_\lambda, \fk s_\lambda \subset\fkn$ by 
\[ \fk l_\lambda = \{ X \in \fkn : \lambda(XY) =0\text{ for all }Y \in \fkn\}\qquad\text{and}\qquad
\fk s_\lambda = \{ X \in \fkn : \lambda(XY) =0\text{ for all }Y\in \fk l_\lambda\}.\]  Alternatively, one constructs the subspace $\fk l_\lambda$ as the left kernel of the bilinear form 
\[\barr{cccc} B_\lambda : &\fkn \times \fkn &\to & \FF_q \\
& (X,Y) & \mapsto & \lambda(XY)\earr
\]
and $\fk s_\lambda$ as the left kernel of the restriction of $B_\lambda$ to the domain $\fkn \times \fk l_\lambda$.  One simultaneously checks the following facts: 
\begin{enumerate}
\item[(a)] The subspace $\fk s_\lambda$ is a subalgebra of $\fkn$.
\item[(b)] The subspace $\fk l_\lambda$ is a right ideal of $\fkn$ and a two-sided ideal of $\fk s_\lambda$.
\end{enumerate}
We therefore may define $L_\lambda\vartriangleleft S_\lambda \subset G$ as the algebra subgroups 
$L_\lambda = 1+\fk l _\lambda$ and $S_\lambda = 1+\fk s_\lambda$.  This notation conforms with our definition of $L_\lambda$ in Section \ref{superchars} because of the following lemma.  

\begin{lemma}\label{5.1}
For each $\lambda \in \fkn^*$ we have 
\[ L_\lambda = \{ g \in G : g\lambda = \lambda\}\qquad\text{and}\qquad S_\lambda = \{ g \in G : g\lambda \in G\lambda \cap \lambda G\}.\]  In particular, $|S_\lambda| / |L_\lambda| = |G\lambda \cap \lambda G| = \langle \chi_\lambda,\chi_\lambda \rangle_G$.
\end{lemma}

\begin{proof}
The characterization of $L_\lambda$ is shown in \cite[Lemma 4.2]{DI}.  By definition, $1+X \in S_\lambda$ if and only if $X \in \ker(\lambda g-\lambda)$ for all $g \in L_\lambda$.  The set $\{ \lambda g-\lambda : g \in L_\lambda\}$ is a subspace of $\fkn^*$ of cardinality $|\lambda L_\lambda|$, and so it follows that  $|S_\lambda| = |G| / |\lambda L_\lambda|$.  Define 
\[ S_\lambda' =  \{ g \in G : g\lambda \in G\lambda \cap \lambda G\}\qquad\text{and}\qquad R_\lambda = \{ g \in G : \lambda g = \lambda\}.\]  Observe that $L_\lambda \subset S'_\lambda$ and so $|S'_\lambda| / |L_\lambda| = |G\lambda \cap \lambda G|$.  
  By \cite[Lemma 4.3]{DI}, $|G|/|L_\lambda\cap R_\lambda| = |G\lambda G|$ and so
\be\label{size} |\lambda L_\lambda| = \frac{|L_\lambda| }{ |L_\lambda \cap R_\lambda| }= \frac{|G|/|L_\lambda \cap R_\lambda|}{|G|/|L_\lambda|} = \frac{|G\lambda G|}{|G\lambda|} = \frac{|G\lambda|}{|G\lambda \cap \lambda G|} = \frac{|G|/|L_\lambda|}{|S_\lambda'|/|L_\lambda|} ={|G|}/{|S_\lambda'|}. \ee  We thus have $|S_\lambda|= |S_\lambda'|$.  On the other hand, if $g = 1 +X \in S_\lambda'$ then $g^{-1}\lambda = \lambda h^{-1}$ for some $h=1+Z \in G$, in which case if $Y \in \fk l_\lambda$ then $\lambda(XY) = (g^{-1}\lambda-\lambda) (Y)  = (\lambda h^{-1}-\lambda) (Y) = \lambda(YZ) = 0$.  Hence $X \in \fk s_\lambda$ so $g \in S_\lambda$, which implies $S_\lambda' \subset S_\lambda$ and by order considerations $S_\lambda' = S_\lambda$.  
 \end{proof}

The algebra subgroups we call $L_\lambda$ and $S_\lambda$ are thus the same as those defined in Section 4 of \cite{AndreAdjoint}, which takes the preceding lemma as a definition.  Also, one sees that $\chi_\lambda$ is fully ramified if and only if $S_\lambda = G$.

%
%
Since  the function $\theta_\lambda : G \to \CC$ restricts to a linear character of $L_\lambda$ and  $\chi_\lambda = \Ind_{L_\lambda}^G(\theta_\lambda)$, it follows by the transitivity of induction that any irreducible constituent of $\Ind_{L_\lambda}^{S_\lambda}(\theta_\lambda)$ becomes a constituent of $\chi_\lambda$ on induction to $G$.  The following theorem asserts the somewhat surprising fact that induction from $S_\lambda$ to $G$ in this situation actually preserves irreducibility.    Andr\'e first observed this phenomenon in the special case $G = \UT_n(q)$ \cite[Theorem 2]{AndreHecke}.  Later, Andr\'e and Nichol\'as generalized this to algebra groups (more precisely, their generalization applies to a family of groups which includes and extends algebra groups) \cite[Theorem 4.2]{AndreAdjoint}.


\begin{theorem}[Andr\'e and Nicol\'as \cite{AndreAdjoint}] \label{andre}  Let $\fkn$ be a finite-dimensional nilpotent $\FF_q$-algebra and write $G=1+\fkn$.   If $\lambda \in \fkn^*$ then the  map \[ \barr{ccc} \Irr\(S_\lambda,\Ind_{L_\lambda}^{S_\lambda}(\theta_\lambda)\) & \to & \Irr(G,\chi_\lambda) \\ \psi & \mapsto & \Ind_{S_\lambda}^{G}(\psi) \earr\] is a bijection. 
\end{theorem}

The idea of the proof is quite simple after observing a few technical details, and we describe a sketch of the one given in \cite{AndreAdjoint}.

\begin{proof}[Sketch of proof]
By Clifford theory (cf. \cite[Theorem 4.1]{AndreAdjoint}), each irreducible constituent $\psi$ of $\Ind_{L_\lambda}^{S_\lambda}(\theta_\lambda)$ appears with multiplicity $\psi(1)$.  It follows from the transitivity of induction and (\ref{decomp}) that $\Ind_{S_\lambda}^{G}(\psi)$ appears as a constituent of the character of the regular representation of $G$ with multiplicity at least $ \frac{|G\lambda G|}{|G\lambda|} \psi(1)= \frac{|G\lambda|}{|G\lambda \cap \lambda G|} \psi(1)= \frac{|G|/|L_\lambda|}{|S_\lambda|/|L_\lambda|}\psi(1) =\Ind_{S_\lambda}^G(\psi)(1) $.   Since the multiplicity of any irreducible character of $G$ in the character of the regular representation is its degree, this identity forces the given map to be well-defined, injective, and surjective. 
\end{proof}

The character $\Ind_{L_\lambda}^{S_\lambda}(\theta_\lambda)$ is a superclass function but not necessarily a supercharacter of $S_\lambda$.  However, it does have an obvious supercharacter as a constituent: namely, the one given by $\Ind_{H}^{S_\lambda}(\theta_\lambda)$ where $H\supset L_\lambda$ is the left stabilizer in $S_\lambda$ of the restricted linear map $\lambda \downarrow \fk s_\lambda \in (\fk s_\lambda)^*$.  This observation readily suggests an inductive method of constructing constituents of $\chi_\lambda$.  We explore this idea in the next section, but first we note a corollary of the preceding theorem concerning the problem of counting the constituents of $\chi_\lambda$.  

To state this we require an additional item of notation.  Given $\lambda \in \fkn^*$ define 
\[ \fk k_\lambda = \fk l_\lambda \cap \ker\lambda\qquad\text{and}\qquad K_\lambda = 1+\fk k_\lambda.\]  By construction $XY \in \fk k_\lambda$ for all $(X,Y) \in \fk l_\lambda \times \fk s_\lambda$ and $\fk s_\lambda \times \fk l_\lambda$, so the subspace $\fk k_\lambda$ is a two-sided ideal in $\fk s_\lambda$.  Therefore $K_\lambda$ is a well-defined normal algebra subgroup of $S_\lambda$, and the quotient $L_\lambda / K_\lambda$ is a central subgroup of $ S_\lambda / K_\lambda$.

These algebra groups relate to the irreducible constituents of $\chi_\lambda$ in the following way.  Here we identify $S_\lambda/ K_\lambda$ and $L_\lambda/K_\lambda$ with algebra groups $1 + \fk s_\lambda / \fk k_\lambda$ and $1 + \fk l_\lambda / \fk k_\lambda$, and let $\pi : 1+X \mapsto 1 + (X+\fk k_\lambda)$ denote the quotient homomorphism $S_\lambda \to S_\lambda / K_\lambda$.

\begin{theorem}\label{count} Let $\fkn$ be a finite-dimensional nilpotent $\FF_q$-algebra, write $G=1+\fkn$, and let $\lambda \in \fkn^*$. 
\begin{enumerate}
\item[(1)] The number of irreducible constituents of the supercharacter $\chi_\lambda$ of $G=1+\fkn$ for $\lambda \in \fkn^*$ is the number of coadjoint $S_\lambda$-orbits in the set $\{ \nu \downarrow \fk s_\lambda : \nu \in \lambda G\}.$

\item[(2)] The number of irreducible constituents of $\chi_\lambda$ of degree $q^e$ is equal to the number of irreducible characters $\psi$ of $S_\lambda/K_\lambda$ such that 
\be\label{condition}
\psi\circ \pi(z) =\frac{|G\lambda \cap \lambda G|}{|G\lambda|} \cdot q^e\cdot   \theta_\lambda(z),\qquad\text{for all }z \in L_\lambda.\ee

\item[(3)] If $\ker\lambda\supset (\fk s_\lambda)^2$, then the number of irreducible constituents of $\chi_\lambda$ of degree $q^e$ is equal to the number of irreducible characters of the quotient group $S_\lambda / L_\lambda$ 
 of degree $\frac{|G\lambda \cap \lambda G|}{|G\lambda|} q^e$.   
\end{enumerate}
\end{theorem}

In fact, in (2) we have not only an equality of counts but a canonical bijection: the characters $\psi \circ \pi$ with $\psi$ ranging over all $\psi \in \Irr(S_\lambda / K_\lambda)$ satisfying our conditions precisely comprise the left-hand set in the bijection of Theorem \ref{andre}.  
As mentioned earlier, a theorem of Isaacs \cite{I95} asserts that the degrees of the irreducible characters of an algebra group over $\FF_q$ are $q$-powers, so the degrees in (2) and (3) are in fact arbitrary.


\begin{proof}
Let $\cO=\{ \nu \downarrow \fk s_\lambda : \nu \in \lambda G\}$.
The equality $|\lambda G | = |G\lambda| = |G|/|L_\lambda|$ along with the definition of $\fk l_\lambda$ and the characterization of $S_\lambda$ in Lemma \ref{5.1} together imply that for $g \in S_\lambda$, \[ \frac{|L_\lambda|}{| G|} \sum_{\nu \in \lambda G} \theta_\nu(g) = \left\{\barr{ll} \theta_\lambda(g),&\text{if }g \in L_\lambda \\ 0,&\text{otherwise}\earr\right.
\qquad\text{and}\qquad
 \Ind_{L_\lambda}^{S_\lambda}(\theta_\lambda)(g) 
=\frac{|S_\lambda|}{|G|}  \sum_{\nu \in \lambda G} \theta_\nu(g).
\]  Since $L_\lambda \subset S_\lambda$, we have $ \{ \nu \in \fkn^* :\nu \downarrow \fk s_\lambda = 0\} \subset  \{ \nu \in \fkn^* :\nu \downarrow \fk l_\lambda = 0\} = \{ \lambda g -\lambda : g\in G\},$ the equality due to \cite[Lemma 4.2]{DI}.  Since the last set is a vector space, it follows that for each $\kappa \in \fk s_\lambda^*$, the set $\{ \nu \in \lambda G : \nu \downarrow \fk s_\lambda = \kappa \}$ is either empty or has cardinality $\frac{|\fkn|}{|\fk s_\lambda|}= \frac{|G|}{|S_\lambda|}$, so we have $\Ind_{L_\lambda}^{S_\lambda}(\theta_\lambda) = \sum_{\nu \in \cO} \theta_\nu$.  One finds using Lemma \ref{5.1} that $\cO$ is a union of two-sided $S_\lambda$-orbits in $(\fk s_\lambda)^*$, and so from (\ref{superchar-def}) we see that  $\Ind_{L_\lambda}^{S_\lambda}(\theta_\lambda) = \sum_{\nu \in \cO} \theta_\nu$ is a sum of supercharacters of $S_\lambda$ and (1) is immediate from Remark (iv) in Section \ref{alg}.

By Theorem \ref{andre} the distinct irreducible constituents of $\chi_\lambda$  are obtained by inducing the irreducible constituents of $\tau\overset{\mathrm{def}}=\Ind_{L_\lambda}^{S_\lambda}(\theta_\lambda)$.  Observe that if $X \in \fk s_\lambda$ then 
\[ \tau(1+X) = \left\{\barr{ll} 0,&\text{if }X \notin \fk l_\lambda, \\
 \tau(1) \cdot \theta\circ \lambda(X), &\text{if }X \in \fk l_\lambda.\earr\right.\]  Thus $g\mapsto \frac{\tau(g)}{\tau(1)}$ is the linear character $\theta_\lambda$ on $L_\lambda$ and $\ker \tau$ contains $K_\lambda$, so by standard theorems in character theory (see \cite[Lemma 2.2]{Isaacs}) every irreducible constituent  of $\tau$
 is of  the form $\psi \circ \pi$ for a distinct irreducible character $\psi$ of $S_\lambda / K_\lambda$.  
 %
 If $\psi \in \Irr(S_\lambda / K_\lambda)$ then the map $g\mapsto \frac{\psi(g)}{\psi(1)}$ defines a linear character of the central subgroup $L_\lambda/K_\lambda$, and so the map $g\mapsto \frac{\psi\circ \pi(g)}{\psi(1)} $ likewise defines a linear character of $L_\lambda = \pi^{-1}(L_\lambda/K_\lambda)$.  Noting 
the formula for $\tau$, it follows that if $\psi \in \Irr(S_\lambda / K_\lambda)$ then  $ \langle \tau, \psi\circ \pi \rangle_{S_\lambda} \neq 0$ if and only if  $\frac{|S_\lambda|}{|G|} \theta_\lambda (z)=  \frac{\tau(z)}{\tau(1)} = \frac{\psi\circ \pi(z)}{\psi(1)}$ for all $z \in L_\lambda$.  
As $\psi\circ\pi$ becomes on induction to $G$ an irreducible character of degree $q^e$ if and only if $\psi(1) = \frac{|S_\lambda|}{|G|} q^e$, and since by Lemma \ref{5.1} we have $\frac{|S_\lambda|}{|G|} =\frac{|S_\lambda|/|L_\lambda|}{|G|/|L_\lambda|} = \frac{|G\lambda \cap \lambda G|}{|G\lambda|}$, part (2) follows.

In the situation of (3),  $\theta_\lambda$ restricts to a linear character of $S_\lambda$, so $\tau$ and $\tau \otimes \overline{\theta_\lambda}$ have the same number of constituents of a given degree.   Since $L_\lambda \vartriangleleft S_\lambda$ and $\tau \otimes \overline{\theta_\lambda} = \Ind_{L_\lambda}^{S_\lambda}(\one)$,  the irreducible constituents of $\tau$ of degree $d$ are in bijection  with the irreducible characters of degree $d$ of the quotient $S_\lambda /L_\lambda$, and (3) follows.
\end{proof}

\begin{example} 
Define 
$\lambda \in \fkt_{13}(2)^*$ by 
\[\lambda = e_{1,5}^* + e_{2,6}^* + e_{3,10}^* + e_{4,11}^* + e_{5,7}^* + e_{6,8}^* + e_{7,9}^* + e_{8,12}^* + e_{9,13}^*.\]  
 The group $\UT_{13}(2)$ has exactly two irreducible characters with non-real values by \cite{IK05}, and it is known that they occur as constituents of degree $2^{16}$ in the supercharacter $\chi_\lambda$. We give constructive proofs of these facts in \cite{supp1}. 
  The two-sided orbit of $\lambda$ has size $2^{39}$, which prohibits or at least discourages one from computing the number of irreducible constituents of $\chi_\lambda$ directly from Remark (iv) in Section \ref{alg}. The size of the set $\{ \nu \downarrow \fk s_\lambda : \nu \in \lambda \cdot \UT_{13}(2)\}$
by contrast is only $2^{16}$.  Using Theorem \ref{count} with a naive implementation of Burnside's lemma, one is therefore able to compute 
$|\Irr(\UT_{13}(2),\chi_\lambda)| = 98.$  
\end{example}

\subsection{From supercharacters $\chi_\lambda$ to the constituents $\xi_\lambda$}\label{inductive}

We are now prepared to define the characters $\xi_\lambda$.  Recall that the supercharacter $\chi_\lambda$ is defined by inducing the function $\theta_\lambda : g \mapsto \theta\circ \lambda(g-1)$ from the algebra subgroup $L_\lambda = 1+\fk l_\lambda$ to $G=1+\fkn$, where $\fk l_\lambda$ is the left kernel of the bilinear form $B_\lambda : (X,Y)\mapsto \lambda(XY)$ on $\fkn$. We define the character $\xi_\lambda$ in the same way, except with $\fk l_\lambda$ replaced by a larger subalgebra $\olfkl_\lambda$, obtained as the left kernel of $B_\lambda$ restricted to the terminal element in a descending chain of subspaces of $\fkn\times \fkn$.  

Here are the details.  For each $\lambda \in \fkn^*$, we define two sequences of subspaces $\fk l_\lambda^i, \fk s_\lambda^i \subset \fkn$ for $i\geq 0$ by the inductive formulas 
\[ \ba \fk l_\lambda ^0 &= 0, \\ 
\fk s_\lambda^0 &= \fkn, \ea \qquad \text{and}\qquad 
\ba \fk l_\lambda^{i+1} & = \left\{ X \in \fk s_\lambda^i : \lambda(XY) = 0 \text{ for all }Y \in \fk s_\lambda^i \right\}, \\ 
\fk s_\lambda^{i+1} & =\left \{ X \in \fk s_\lambda^i : \lambda(XY) = 0 \text{ for all }Y \in \fk l_\lambda^{i+1}\right \}. \ea\]
If $B_\lambda : \fkn\times \fkn \to \FF_q$ denotes the bilinear form $(X,Y) \mapsto \lambda(XY)$, then we may alternatively define the subspaces $\fk l_\lambda^i$, $\fk s_\lambda^i$ for $i>0$ by
\[ \ba \fk l_\lambda^i & =\text{the left kernel of the restriction of $B_\lambda$ to $\fk s_\lambda^{i-1}\times \fk s_{\lambda}^{i-1}$,} \\
\fk s_\lambda^i &=\text{the left kernel of the restriction of $B_\lambda$ to $\fk s_\lambda^{i-1} \times \fk l_\lambda^i$.}\ea\]
In particular, $\fk l_\lambda^1 = \fk l_\lambda$ and $\fk s_\lambda^1 = \fk s_\lambda$.

It follows by repeatedly applying Lemma \ref{5.1} and the observations preceding it that
we have an ascending and descending chain of subspaces
\be\label{chain} 0  =\fk l_\lambda^0 \subset  \fk l_\lambda^1 \subset \fk l_\lambda^2 \subset \cdots\subset \fk s_\lambda^2 \subset \fk s_\lambda^1  \subset \fk s_\lambda^0 =  \fkn\ee with the following properties:
\begin{enumerate}
\item[(a)] Each $\fk s_\lambda^{i+1}$ is a subalgebra of $\fk s_\lambda^i$.
\item[(b)] Each $\fk l_\lambda^{i+1}$ is a right ideal of $\fk s_\lambda^i$ and a two-sided ideal of $\fk s_\lambda^{i+1}$.
\item[(c)] If $\fk s_\lambda^{n-1} = \fk s_\lambda^{n}$ for some $d\geq 1$ then $\fk l_\lambda^{d+i} = \fk l_\lambda^{d}$ and $\fk s_\lambda^{d+i} = \fk s_\lambda^d$ for all $i\geq 0$.
\end{enumerate}
Let $d\geq 1$ be an integer such that (c) holds$-$by dimensional considerations, some such $d$ exists$-$and define the subalgebras $\olfkl_\lambda,\olfks_\lambda\subset\fkn$ and algebra subgroups $\oll_\lambda,\ols_\lambda\subset G$ by \[  
\olfkl_\lambda = \fk l_\lambda^{d},
\qquad
 \olfks_\lambda = \fk s_\lambda^{d}\qquad\text{and}
 \qquad 
 \oll_\lambda = 1+\olfkl_\lambda,
 \qquad
 \ols_\lambda = 1+\olfks_\lambda.\] 
Observe that $\olfkl_\lambda$ and $\olfks_\lambda$ are just the terminal elements in the ascending and descending chains (\ref{chain}), and hence by Lemma \ref{5.1} the restriction $\lambda \downarrow \olfks_\lambda \in (\olfks_\lambda)^* $ is fully ramified. 

It is clear that the function $\theta_\lambda: G \to \CC$ restricts to a linear character of $\oll_\lambda$ since $\lambda(XY) =0$ for $X,Y \in \fk l_\lambda^i$ for all $i$.  
We may thus define a character of $G$ by
\[\xi_\lambda = \Ind_{\oll_\lambda}^G(\theta_\lambda).\]
This character is our principle object of study.  
 It is a 
 possibly reducible constituent  with  degree $|G|/|\oll_\lambda|$  of the supercharacter $\chi_\lambda$ of $G$;  also, we have this immediate corollary of Theorem \ref{andre}:

\begin{corollary}\label{andre-cor}
Let $\fkn$ be a finite-dimensional nilpotent $\FF_q$-algebra and write $G=1+\fkn$.   If $\lambda \in \fkn^*$ then 
%
the map \[ \barr{ccc} \Irr\(\ols_\lambda,\Ind_{\oll_\lambda}^{\ols_\lambda}(\theta_\lambda)\) & \to & \Irr(G,\xi_\lambda) \\ \psi & \mapsto & \Ind_{\ol S_\lambda}^{G}(\psi) \earr\] is a bijection.   

\end{corollary}

Observe that whereas in Theorem \ref{andre} the character $\Ind_{L_\lambda}^{S_\lambda} (\theta_\lambda)$ was \emph{not} necessarily a supercharacter of $S_\lambda$, here the character $\Ind_{\oll_\lambda}^{\ols_\lambda}(\theta_\lambda)$ is a fully-ramified supercharacter of $\ols_\lambda$ by construction. 
Thus the problem of enumerating the irreducible characters of an algebra is group is reduced to examining certain fully ramified supercharacters of the algebra subgroups $\ols_\lambda$.  Furthermore, this reduction is easily computable and it can tell us useful information about the characters of $G$ besides just their number and degrees.  In this sense the characters $\xi_\lambda$ offer a sort of middle ground between the algebra group character reduction described by Evseev in \cite{E}, which allows us to compute similar data but provides a much less tangible description of the characters involved, and the reduction process put forth by Boyarchenko in \cite{Boy}, which specifies several nice structural features of the irreducible representations of $G$ but is highly nonconstructive.

A less immediate corollary of Theorem \ref{andre} is the next proposition.

\begin{proposition}\label{less}
Let $\lambda,\lambda' \in \fkn^*$ and write $\mu = \lambda\downarrow \olfks_\lambda$ and $\mu' =\lambda'\downarrow \olfks_\lambda$.  Then $\mu' \in \mu \ols_\lambda$ if and only if $\lambda' \in \lambda \ols_\lambda$, and in this case the following hold:
\begin{enumerate}
\item[(1)] $\Ind_{\ols_\lambda}^G(\psi_{\mu'}) = \psi_{\lambda'}$, and consequently $\langle \xi_\lambda, \psi_{\lambda'} \rangle_G$ is a positive integer.
\item[(2)] $\psi_{\lambda'} \in \Irr(G)$ if and only if  $\psi_{\mu'} \in \Irr(\ols_\lambda)$.
\end{enumerate}

\end{proposition}

\begin{proof}  Abbreviate by writing $\fk l = \olfkl_\lambda$ and $\fk s= \olfks_\lambda$ and $L = \oll_\lambda$ and $S = \ols_\lambda$.  
By construction, $\fk s \subset \ker(\lambda'-\lambda)$ if and only if $\lambda'-\lambda\in \fkn^*$ is a map of the form $X \mapsto \lambda(XY)$ for some $Y \in \fk l$, and in this case $\lambda'-\lambda = \lambda y -\lambda$ for $y =(1+Y)^{-1}\in L$.   Thus $\mu' = \mu$ if and only if $\lambda' \in \lambda L$, so $\mu' \in \mu S$ if and only if $\lambda' \in \lambda S$ since $L \subset S$.  

Assume $\mu' \in \mu S$ and let $\vartheta_1,\dots,\vartheta_m$ be the distinct irreducible constituents of the supercharacter $\chi_\mu$ of $S$.  
We have $ \chi_\mu= \chi_{\mu'}$ so by Remark (iv) in Section \ref{alg} there are complex numbers $c_i$ such that $\psi_{\mu'} = \sum_{i=1}^m c_i \vartheta_i$.  The preceding corollary shows that induction from $S$ to $G$ defines a bijection from the irreducible constituents of $\chi_\mu$ to those of $\xi_\lambda$, so we have 
\[\barr{c} 1 =  \langle \psi_{\mu'}, \psi_{\mu'} \rangle_{S} =  \sum_{i=1}^m |c_i|^2 =
\left \langle \Ind_{S}^G(\psi_{\mu'}), \Ind_{S}^G(\psi_{\mu'})\right \rangle_G.
\earr\]
 By \cite[Theorem 5.5]{AndreAdjoint}, $\Ind_{S}^G(\psi_{\mu'})$ is equal to $\psi_{\lambda'}$ plus a linear combination of Kirillov functions with nonnegative integer coefficients; given this, the preceding equation implies that $\Ind_{S}^G(\psi_{\mu'}) = \psi_{\lambda'}$. 
Corollary \ref{andre-cor} likewise implies that $ \langle \chi_\mu,\psi_\mu'\rangle_S = \langle \xi_\lambda, \Ind_{S}^G(\psi_{\mu'}) \rangle_G$, and \cite[Theorem 5.7]{AndreAdjoint} shows that first of these numbers is a positive integer.  This proves (1).  For (2), we note that  $\psi_{\lambda'} = \sum_{i=1}^m c_i\hspace{0.5mm} \Ind_{S}^G (\vartheta_i)$ is 
 a character of $G$ if and only if every $c_i$ is a nonnegative integer, which occurs if and only if $\psi_{\mu'}$ is a character of $S$. 
 \end{proof}

For each $\lambda \in \fkn^*$, let 
\[ \Xi_\lambda =\left\{ g\lambda sg^{-1} : g \in G,\ s \in \ols_\lambda\right\} \subset \fkn^*.\]  
This is certainly a union of coadjoint orbits, and it is evident from the preceding proposition that a coadjoint orbit $\cO \subset \fkn^*$ has $\cO \subset \Xi_\lambda$ if and only if $\cO$ contains an element whose restriction to $\olfks_\lambda$ lies in the orbit $\mu \ols_\lambda$ where $\mu = \lambda \downarrow \olfks_\lambda$.  

  Noting that $\olfks_\lambda \subset \fk s_\lambda^i$ and that $\fk l_\lambda^i$ is a right ideal of $\olfks_\lambda$ for all $i$, it is not difficult to show by inducting upward on the superscripts in the chain (\ref{chain}) that
\begin{enumerate}
\item[(a)] If $s \in \ols_\lambda$ then $\ols_{\lambda s} = \ols_\lambda$.
\item[(b)] If $g \in G$ then $\ols_{g\lambda g^{-1}} = g\cdot  \ols_\lambda \cdot g^{-1}$.
\end{enumerate}
Consequently, $\Xi_\mu = \Xi_\lambda$ if $\mu \in \Xi_\lambda$ and the sets $\Xi_\lambda$ thus partition $\fkn^*$.  We now may state a theorem enumerating the the various significant structural properties of the characters $\xi_\lambda$.

\begin{theorem}\label{structural}
Let $\fkn$ be a finite-dimensional nilpotent $\FF_q$-algebra and write $G=1+\fkn$. If $\lambda,\lambda' \in \fkn^*$, then 
\begin{enumerate}
\item[(1)]  
$\displaystyle \xi_\lambda =\frac{|\ols_\lambda|}{\left|G\right|} \sum_{\nu \in \Xi_\lambda} \theta_\nu 
$ and $\displaystyle |\Xi_\lambda| = \frac{|G|^2}{| \oll_\lambda|| \ols_\lambda|}. $    
\item[(2)] $\xi_\lambda = \xi_{\lambda'}$ if and only if $\lambda' \in \Xi_\lambda$, and if $\lambda ' \notin \Xi_\lambda$ then $\xi_\lambda$ and $\xi_{\lambda'}$ share no irreducible constituents.   In particular, 
\[ \left\langle \xi_\lambda, \xi_{\lambda'} \right\rangle_G = \begin{cases} |\ols_\lambda| / |\oll_\lambda|,&\text{if }\lambda' \in \Xi_\lambda, \\
0,&\text{otherwise}.\end{cases}\] 

\item[(3)] If the character $\xi_\lambda$ is irreducible then $\xi_\lambda = \psi_\lambda$ is a Kirillov function.

\item[(4)] The irreducible constituents of the characters $\{ \xi_\lambda : \lambda \in \fkn^*\}$ partition $\Irr(G)$, and the character $\rho_G$ of the regular representation of $G$ decomposes as 
\[ \rho_G = \sum_\nu \frac{|G|}{|\ols_\nu|} \xi_\nu\] where the sum is over any set of $\nu \in \fkn^*$ for which $\fkn^* = \bigcup_\nu \Xi_\nu$ is a disjoint union.


\end{enumerate}
\end{theorem}

\begin{remarks}
\begin{enumerate}

\item[]

\item[(i)]
This result extends Theorem 4.9 in \cite{AndreAdjoint}, which asserts the existence of a set of characters $\xi_\lambda$ whose constituents partition $\Irr(G)$ and which are induced from fully ramified supercharacters of algebra subgroups, but does not provide a way of computing various data attached to them.

\item[(ii)] While the characters $\{\xi_\lambda\}$ have several nice properties in common with the supercharacters $\{ \chi_\lambda\}$ of $G$ (e.g., they are characters with disjoint constituents; they have a formula given by a sum of functions $\theta_\nu$ over a kind of orbit; they are Kirillov functions if irreducible), they do \emph{not} form a supercharacter theory in the sense of Diaconis and Isaacs \cite{DI}.  Notably, the subspace of class functions spanned by $\{\xi_\lambda\}$ is not closed under tensor products, as one can check by examining the algebra group of matrices $g \in \UT_6(2)$ with $g_{1,2} = g_{5,6}$.  Furthermore, if $H =1+\h$ is an algebra subgroup of $G=1+\fkn$ then $\Res_H^G(\xi_\lambda)$ for $\lambda \in \fkn^*$ (respectively, $\Ind_H^G(\xi_\mu)$ for $\mu \in \h^*$) may fail to  be a linear combination of the characters $\{ \xi_\mu : \mu \in \h^*\}$ (respectively, the characters  $\{ \xi_\lambda :\lambda \in \fkn^*\}$).

\end{enumerate}
\end{remarks}

\begin{proof}
Abbreviate by writing $\fk s = \olfks_\lambda$ and $S = \ols_\lambda$ and $L = \oll_\lambda$.
Let $\mu = \lambda \downarrow \fk s$ and recall that $\chi_\mu$ is a fully ramified supercharacter of $S$, and therefore equal to the sum $\sum_{\nu \in \cR} \psi_{\nu}(1) \psi_{\nu}$  where $\cR \subset \fk s^*$ is a set of representatives of the distinct coadjoint orbits in $\mu S$.  For each $\nu \in \fk s^*$ choose $\wt\nu  \in \fkn^*$ with $\wt\nu\downarrow \fk s = \nu$.  The preceding proposition asserts $\Ind_S^G(\psi_{\nu}) = \psi_{\wt\nu}$ for each $\nu \in \cR$, and this implies that $\psi_{\wt\nu}(1) = \frac{|G|}{|S|} \psi_{\nu}(1)$.  Noting our remarks following the definition of $\Xi_\lambda$, one sees that each coadjoint orbit in $\Xi_\lambda$ contains $\wt\nu$ for a unique $\nu \in \cR$, and conversely that $\wt\nu \in \Xi_\lambda$ for all $\nu \in \cR$.    
From these observations, we deduce that
\[ \xi_\lambda = \Ind_{S}^G(\chi_\mu) = \frac{|S|}{|G|}\sum_{\nu  \in \cR} \psi_{\wt\nu}(1) \psi_{\wt\nu} = \frac{|S|}{|G|} \sum_{\nu \in \Xi_\lambda} \theta_\nu,\] and  this formula gives $\frac{|G|}{|L|} = \xi_\lambda(1) = \frac{|S|}{|G|} |\Xi_\lambda|$, which proves (1).
Since the functions $\{\theta_\lambda : \lambda \in \fkn^*\}$ are orthonormal with respect to $\langle\cdot,\cdot\rangle_G$, our formula for $\langle \xi_\lambda,\xi_{\lambda'}\rangle_G$ is immediate from (1).  If $L = S$ then $\chi_\mu = \psi_\mu$ is a linear character, whence $\xi_\lambda = \psi_\lambda$ by the preceding proposition.  Finally,  (4) holds because $\rho_G = \sum_{\nu \in \fkn^*} \theta_\nu$ and the sets $\Xi_\nu$ partition $\fkn^*$.
\end{proof}

We observed in the preceding proof that the number of coadjoint $G$-orbits in $\Xi_\lambda$ is the number coadjoint $\ols_\lambda$-orbits in $\mu \ols_\lambda$.  Therefore  invoking Theorem \ref{count} and Corollary \ref{andre-cor} provides the following result.

\begin{corollary} The number of irreducible constituents of $\xi_\lambda$ for $\lambda \in \fkn^*$ is equal to the number of coadjoint $G$-orbits in $\Xi_\lambda$, and also to the number of coadjoint $\ols_\lambda$-orbits in $\mu \ols_\lambda$ where $\mu = \lambda \downarrow \olfks_\lambda$.
\end{corollary}

There are of course analogues for the other parts of Theorem \ref{count} whose formulation we leave to the reader.

%
%
%
%
%
%

\section{Well-induced characters and Kirillov functions}

It happens that $\xi_\lambda$ is irreducible if and only if $\xi_\lambda = \psi_\lambda = \logpsi_\lambda$. 
This fact manifests a more general phenomenon relating Kirillov functions and induced linear characters, which we devote this section to discussing.  

\subsection{Inflation from quotients by algebra subgroups}\label{infl-sec}

\def\q{\mathfrak{q}}

We very briefly discuss  the effect of inflation on the functions $\psi_\lambda$, $\logpsi_\lambda$, $\chi_\lambda$, $\xi_\lambda$.  
Suppose $\fkn$ is a nilpotent $\FF_q$-algebra with a two-sided ideal $\h$.  Let $\q=\fkn /\h$ be the quotient algebra, and write $G = 1+\fkn$ and $Q = 1+\q$.  Also, let $\wt \pi : \fkn \to \q$ be the quotient map, and define $\pi :G\to Q$ by $\pi(1+X) = 1+\wt\pi(X)$; both of these are surjective homomorphisms, of algebras and groups, respectively.
The following result is really very elementary, but useful to have written down.

\begin{observation} \label{inflation}  If $\lambda \in \fkn^*$ has $\ker \lambda \supset \h$, then there exists a unique $\mu \in \q^*$ with $\lambda = \mu\circ \wt\pi$, and 
\[\psi_\lambda = \psi_{\mu} \circ \pi,
\qquad
\logpsi_\lambda = \logpsi_{\mu} \circ \pi,
\qquad 
\chi_\lambda = \chi_{\mu } \circ\pi,
\qquad\text{and}
\qquad 
\xi_\lambda = \xi_{\mu} \circ \pi.\] Furthermore, $\oll_\lambda = \pi^{-1}(\oll_{\mu})$ and $\ols_\lambda = \pi^{-1}(\ols_\mu) $.
\end{observation}

\begin{remark} Since $\chi \mapsto \chi\circ \pi$ defines an injection $\Irr(Q)\to \Irr(G)$, by  \cite[Lemma 2.22]{Isaacs} for example, if $\psi_\mu$ or $\logpsi_\mu$ are characters in this setup then the same is true of $\psi_\lambda$ or $\logpsi_\lambda$, respectively.
\end{remark}

\begin{proof}
For each $\lambda \in \fkn^*$ with $\ker \lambda \supset \h$, let $\lambda^\pi \in \q^*$ be the map with $\lambda^\pi (X+\fk q) = \lambda(X)$ for $X \in \fkn$.  Then $\mu = \lambda^\pi$ is clearly the unique element of $\q^*$ with $\lambda = \mu \circ \wt \pi$.  Furthermore, our hypotheses imply that  $\ker g\lambda h \supset \h$ and $(g\lambda h)^\pi = \pi(g)\cdot \lambda^\pi \cdot \pi(h)$ for all $g,h \in G$; this is a straightforward exercise and also a consequence of \cite[Proposition 3.1]{M1}.  
It follows that the map $\nu \mapsto \nu^\pi$ gives bijections $\lambda^G \to \mu^Q$ and $G\lambda G \to Q\mu Q$ and $G\lambda\to Q\mu$, so the formulas (\ref{superchar-def}) and (\ref{formal}) imply that $\psi_\lambda = \psi_\mu\circ \pi$ and $\logpsi_\lambda = \logpsi_\mu\circ \pi$ and $\chi_\lambda = \chi_\mu\circ \pi$.

We have $\lambda(XY) = 0$ for all $(X,Y) \in \fkn\times \h$ or $\h\times \fkn$ 
since $\h$ is a two-sided ideal contained in $\ker(\lambda)$, and so it follows easily by induction that $\fk l_\lambda^i = \pi^{-1}(\fk l_\mu^i )$ and $\fk s_\lambda^i = \pi^{-1}(\fk s_\mu^i)$ for all $i$, which suffices to show $\oll_\lambda = \pi^{-1}(\oll_{\mu})$ and $\ols_\lambda = \pi^{-1}(\ols_\mu)$.
 Since $\theta_\lambda = \theta_\mu \circ \pi$, 
 the identity $\xi_\lambda = \xi_\mu\circ \pi$ is now a consequence of  the more general fact that if $\pi : G \to Q$ is a surjective group homomorphism and $\psi : L\to \CC$ is a function on a subgroup $L\subset Q$, then $\Ind_{\pi^{-1}(L)}^G (\psi\circ \pi) = \Ind_L^Q(\psi)\circ \pi$; this statement in turn follows easily from the Frobenius formula for induction (\ref{frob}).
\end{proof}

As one application of this statement, we can prove the following corollary to Theorem \ref{count}.  Recall from Section \ref{fk} the definitions of the nilpotent $\FF_q$ algebras $\fk k_\lambda$, $\fk s_\lambda$ and algebra groups $K_\lambda$, $S_\lambda$ for $\lambda \in \fkn^*$.  

\begin{corollary}  
Let $\fkn$ be a finite-dimensional nilpotent $\FF_q$-algebra, write $G=1+\fkn$, and choose $\lambda \in \fkn^*$.  
Then 
every irreducible
constituent  with degree $q^e$ of the supercharacter $\chi_\lambda$ is a Kirillov function if and only if every irreducible character satisfying (\ref{condition}) of the algebra group $S_\lambda / K_\lambda \cong 1 + \fk s_\lambda / \fk k_\lambda$  is a Kirillov function.

%
\end{corollary}

\begin{remark}
Using the results in the next section, one can adapt our proof to show that the corollary also holds with exponential Kirillov functions in place of ordinary Kirillov functions.
\end{remark}

\begin{proof} Write $\pi : S_\lambda \to S_\lambda/K_\Lambda$ for the quotient homomorphism $\pi : 1+X \mapsto 1 + (X+\fk k_\lambda)$ as in Theorem \ref{count}.   Write $\fk X$ for the subset of $\Irr(S_\lambda/K_\lambda)$ consisting of all irreducible characters $\psi$ satisfying the condition (\ref{condition}) for some integer $e$.  
Theorem \ref{andre} and the proof of Theorem \ref{count} show that the following composed map defines a bijection from $\fk X$ to the set of irreducible constituents of $\chi_\lambda$ with degree $q^e$: \[ \barr{ccccc} \fk X& \xrightarrow{\text{Inflation}} &  \Irr\(S_\lambda,\Ind_{L_\lambda}^{S_\lambda}(\theta_\lambda)\) & \xrightarrow{\text{Induction}} & \Irr(G,\chi_\lambda) \\
\psi & \mapsto & \psi \circ \pi & \mapsto & \Ind_{S_\lambda}^{G} ( \psi\circ \pi) \earr
\] The preceding observation shows that $\psi \in \fk X$ is a Kirillov function if and only if $\psi \circ \pi$ is a Kirillov function.    By an argument similar to one in the proof of Proposition \ref{less}, Theorem \ref{andre} implies in turn that $\psi\circ \pi$ is a Kirillov function if and only if $\Ind_{S_\lambda}^G(\psi\circ \pi)$ is a Kirillov function.  \end{proof}

\subsection{Well-induced characters}

Following Evseev in \cite{E}, we say that a character $\chi$ of an algebra group $G = 1+\fkn$ is \emph{well-induced} if there exists a subalgebra $\h \subset \fkn$ and a linear character $\tau$ of $H=1+\h$ such that $\chi = \Ind_H^G(\tau)$ and $1+\h^2 \subset \ker \tau$.  In contradistinction to Evseev's definition, we do not require $\chi$ to be irreducible.
Supercharacters $\chi_\lambda$ and the characters $\xi_\lambda$ are then both well-induced.  Since by a theorem  of Halasi \cite{H} every irreducible character of an algebra group $G$ is induced from a linear character of an algebra subgroup, one naturally asks which elements of $\Irr(G)$ are well-induced.  In partial answer, Evseev describes an algorithm to construct certain well-induced irreducible characters of $G$ in the paper \cite{E}, which he implements for $G = \UT_n(q)$.  

This question turns out to be closely related to the problem of determining when a Kirillov function is a character, which one begins to see with the following observation.  
\begin{observation}
If $\h$ is a nilpotent $\FF_q$-algebra and 
$\tau$ is a linear character of $H=1+\h$, then $1+\h^2 \subset \ker \tau$ if and only if  $\tau = \theta_\lambda$ for some $\lambda \in \h^*$.
\end{observation}


\begin{proof}
If $X \in \h$ and $Y \in \h^2$ then $1+X+Y = (1+X)(1+Z)$ for $Z=(1+X)^{-1}Y \in \h^2$, so 
the condition $1+\h^2 \subset \ker \tau$ implies that  $ \tau(1+X+Y)  =  \tau(1+X+Y+XY)=\tau(1+X)  \tau(1+Y) $ for $X,Y \in \h$.
Thus the function $X \mapsto \tau(1+X)$ is an irreducible character of $\h$ viewed as an additive group so we must have $\tau = \theta_\lambda$ for some $\lambda \in \h^*$. 
If $\theta_\lambda$ is a linear character, conversely, then $\theta_\lambda(gh) = \theta_\lambda(g)\theta_\lambda(h)$ for all $g,h \in H$ implies $\theta_\lambda(1+XY) = \theta\circ\lambda(XY) = 1$ for all $X,Y \in \h$ so $1+\h^2 \subset \ker \theta_\lambda$.
\end{proof}
As noted in the proof, in order for $\theta_\lambda$ to be a linear character of $H$, we must have $\theta\circ\lambda(XY) = 1$  for all $X,Y \in \h$, and this suffices to show that $H\lambda H = \{\lambda\}$ and $\theta_\lambda = \chi_\lambda$. 
We may thus equivalently define a well-induced character as a character induced from a linear supercharacter of an algebra subgroup.  (Recall that when a supercharacter $\chi_\lambda$ is linear, it is also a Kirillov function; in particular, $\chi_\lambda = \psi_\lambda = \xi_\lambda$.)

 In what follows, we let $\cF : \fkn \to G$ denote a polynomial map with constant term 1 and linear term $X$; i.e., a bijection of the form 
\be\label{polynomial} \cF(X) = 1 + X + \sum_{k\geq 2} a_k X^k,\qquad\text{with }a_k \in \FF_q,\ee as in the statement of Observation \ref{obs1}.   
We recall from (\ref{formal}) that if $\psi : G\to \CC$ then $\psi^\cF$ is the function on $G$ defined by the identity 
\[\psi^\cF\( \cF(X)\) = \psi(1+X),\qquad\text{for }X \in \fkn.\] As in the introduction, we call the functions $\psi_\lambda^\cF$ with $\cF$ of this form \emph{polynomial Kirillov functions}.  
A polynomial map $\cF$ defines a bijection $\fkn \to G$ if and only if $\cF$ has constant term 1 and nonzero linear term.  We impose the additional condition that $\cF$ have linear term $X$ to ensure that $\chi_\lambda^\cF = \chi_\lambda$ for all $\lambda \in \fkn^*$.

Since $\cF(gXg^{-1}) = g\cdot \cF(X)\cdot g^{-1}$ for $g \in G$ and $X \in \fkn$, it follows from our standard induction formula that if $H\subset G$ is an algebra subgroup and $\psi : H\to \CC$, then $\Ind_H^G(\psi)^\cF = \Ind_H^G(\psi^\cF)$.
This observation gives the following proposition.

\begin{proposition}\label{well} If $\chi$ is a well-induced character of the algebra group $G=1+\fkn$, then $\chi= \chi^\cF$ and $\chi$ is irreducible if and only if $\chi$ is a Kirillov function $\chi = \psi_\lambda$ for some $\lambda \in \fkn^*$.
%
%
\end{proposition}

\begin{proof}
Let $\h$ be a subalgebra of $\fkn$, write $H=1+\h$, and choose a linear character $\tau$  of $H$ with $\chi = \Ind_H^G(\tau)$.  We have $\tau\( \cF(X) \)= \tau(1+X)$ for any $X \in \h$ by the proof of the observation above, so $\chi=\chi^\cF$ by the discussion preceding the proposition statement.

By Remark (iii) in Section \ref{alg}, inducing a Kirillov function from an algebra subgroup yields a linear combination of Kirillov functions with nonnegative integer coefficients.  Since Kirillov functions are orthonormal, a well-induced character  is thus irreducible only if it is equal to a Kirillov function.  Conversely, if a well-induced character $\psi$ is equal to Kirillov function then we have $\langle \psi, \psi\rangle_G =1$.    
\end{proof}

 Hence if $\psi_\lambda$ is a well-induced character then $\psi_\lambda = \logpsi_\lambda$, and by Theorem \ref{structural}  we see that $\xi_\lambda$ is irreducible if and only if $\xi_\lambda=\psi_\lambda=\logpsi_\lambda$.  It can happen that $\psi_\lambda$ is a character without being well-induced, as in the following example.
 
\begin{example} The algebra group $G=1+\fkn$ over $\FF_2$  consisting of all matrices of the form
\[\(\barr{cccc} 1 & a & b & c \\ & 1 & 0 & a+b \\ & & 1 & b \\ & & & 1 \earr\),\qquad\text{for }a,b,c \in \FF_2\]  is isomorphic to the group of quaternions $Q_8$.  If  $\lambda \in \fkn^*$ is defined by $\lambda(X) =X_{1,4}$, then the unique nonlinear irreducible character of $G$ is not well-induced yet is equal to the Kirillov function $\psi_\lambda$.  All the irreducible characters of this group are Kirillov functions and integer-valued, in fact, yet $ \xi_\lambda=\chi_\lambda=2\psi_\lambda$ is a reducible, fully ramified supercharacter.  
\end{example}

It can also happen that the Kirillov function $\psi_\lambda$ is a well-induced, irreducible character but $\psi_\lambda \neq \xi_\lambda$, as the next example illustrates.

\begin{example}  \label{?}
Assume $\FF_q$ has odd characteristic, and let  $G=1+\fkn$ and $H=1+\h$ where $\h \subset \fkn \subset \fkt_5(q)$ are the subalgebras 
\[ \fkn = \left\{ \(\barr{ccccc} 0 & a & * & * & * \\ & 0 & b & * & * \\ & & 0 & -b & * \\ & & & 0 & a \\ & & & & 0\earr\)\right \} \qquad \text{and}\qquad 
\h = \left\{ \(\barr{ccccc} 0 & a & * & * & * \\ & 0 & 0 & * & * \\ & & 0 & 0 & * \\ & & & 0 & a \\ & & & & 0\earr\) \right\}.
 \]  If $\lambda \in \fkn^*$ is defined by $\lambda(X) = X_{1,3}  + X_{2,4} + X_{3,5}$, then $\xi_\lambda =\chi_\lambda$ is fully ramified with degree $q^2$.  However, $\theta_\lambda$ restricts to a linear character of $H$ and $\psi_\lambda = \Ind_H^G(\theta_\lambda)$.  Thus $\psi_\lambda$ is a well-induced, irreducible character with degree $q$, whence we get $\xi_\lambda=\chi_\lambda = q \psi_\lambda$.
\end{example}

Evseev, using the algorithm he describes in \cite{E}, has carried out an impressive computer calculation showing that 
\begin{enumerate}
\item[$\bullet$] If $n\leq 12$ then every irreducible character of $\UT_n(q)$ is well-induced. 

\item[$\bullet$]$\UT_{13}(q)$ has $q(q-1)^{13}$ irreducible characters which are not-well induced \cite[Theorem 1.4]{E}.  
\end{enumerate}
It follows from the preceding proposition and the discussion in Section \ref{infl-sec} that every irreducible character of $\UT_n(q)$ has the form $\psi_\lambda = \logpsi_\lambda$ if $n \leq 12$, and that 
$\UT_n(q)$ has an irreducible character which is not a Kirillov function for $n>12$.  Said in other words:

\begin{theorem} \label{evseev} Let $q>1$ be any prime power.  Then the set of irreducible characters of $\UT_n(q)$ is equal to the set of Kirillov functions of $\UT_n(q)$ if and only if $n\leq 12$.
\end{theorem}

\begin{remark}
Aside from our simple observation connecting the dots, the proof of this result derives entirely from Evseev's formidable computer calculations. 
We give a constructive, computation-free proof of the ``only if'' part of this theorem in \cite{supp1}.  A simple proof  
 which  does not require the use of a computer remains to be found for the ``if'' direction, however.
\end{remark}

In light of Example \ref{?}, it does not follow that $\Irr\(\UT_n(q)\) = \{\xi_\lambda : \lambda \in \fkt_n(q)^*\}$ for $n\leq 12$. 
Nevertheless, this does hold for $q=2$ (as we have checked using a computer) and we see no reason why it should not hold in general.  We therefore make the following conjecture, which probably can be verified in a way analogous to the proofs of Evseev's computational results in \cite{E}.

\begin{conjecture}
If $n\leq 12$ then $\psi_\lambda = \logpsi_\lambda = \xi_\lambda \in \Irr\(\UT_n(q)\)$ for all $\lambda \in \fkt_n(q)^*$.
\end{conjecture}

\begin{remark}
We reiterate that for $n\geq 13$, it is known that for all prime powers $q$ there exists $\lambda \in \fkt_n(q)$ such that the Kirillov function $\psi_\lambda$ is not a character, which implies that $\psi_\lambda\neq \xi_\lambda$.
\end{remark}

Our next aim is to show that Remark (iv) in Section \ref{alg} holds in greater generality.  Specifically, we will extend a theorem of Otto \cite{O} to show that the irreducible constituents of $\xi_\lambda$ span $\psi_\lambda^\cF$ for any appropriate polynomial map $\cF$, in particular for $\cF=\exp$.  
We begin by proving the lemma which will take the place of \cite[Lemma 2.4]{O} in Otto's proof.   Continue to assume that $\cF$ is a polynomial map of the form (\ref{polynomial}).

\begin{lemma} If $\lambda \in \fkn^*$ and $\eta \in \Irr(G)$ is a linear character, then $\langle \chi_\lambda, \eta\rangle_G=0$ implies $\langle \psi_\lambda^\cF,\eta \rangle_G = 0$.  
\end{lemma}

\begin{proof}
Let $\fk s = \olfks_\lambda$ and $S = \ols_\lambda$ and write $\mu = \lambda \downarrow \fk s \in \fk s^*$.  By Proposition \ref{less} we have $\psi_\lambda = \Ind_S^G(\psi_\mu)$ and so by the discussion above  $\psi_\lambda^\cF = \Ind_S^G(\psi_\mu^\cF)$.  By Frobenius reciprocity, 
\[ \left\langle \psi_\lambda^\cF,\eta \right\rangle_G = \left\langle \psi_\mu^\cF,\Res_S^G(\eta) \right\rangle_S
\qquad\text{and}\qquad
\left\langle \chi_\lambda,\eta \right\rangle_G \geq \left\langle \xi_\lambda,\eta \right\rangle_G = \left\langle \chi_\mu,\Res_S^G(\eta) \right\rangle_S \geq 0.\]
Thus $\left\langle \chi_\lambda,\eta \right\rangle_G=0$ implies $\left\langle \chi_\mu,\Res_S^G(\eta) \right\rangle_S =0$ and $\left\langle \psi_\mu^\cF,\Res_S^G(\eta) \right\rangle_S=0$ implies $\left\langle \psi_\lambda^\cF,\eta \right\rangle_G=0$.  Since  $\Res_S^G(\eta) \in \Irr(S)$ is a linear character and $\mu \in \fk s^*$ is fully-ramified, it suffices to prove the lemma when $\lambda \in \fkn^*$ is fully ramified.  
%

We thus assume $G\lambda G=G\lambda =\lambda G$ so that  $\fkn = \fk s_\lambda$ and $\fk l_\lambda$ is a two-sided ideal in $\fkn$.  By definition, we then have $\lambda(XY) =\lambda(YX)= 0$ for all $X \in \fk l_\lambda$ and $Y \in \fkn$.  Since $\eta$ is a class function and polynomials commute with conjugation,  
\be\label{thata}\ba  \left\langle \psi_\lambda^\cF,\eta\right \rangle_G &
= \frac{\sqrt{|\lambda^G|}}{|G|}\sum_{X \in \fkn} \theta\circ \lambda(X) \cdot \overline{\eta\(\cF(X)\)} = \sqrt{|\lambda^G|}\cdot\left  \langle \theta_\lambda^\cF, \eta\right\rangle_G 
.\ea\ee
Observe that if $X,Y \in \fkn$ then $\cF(X+Y) = \cF(Y) + X+ f(X,Y)$ for a polynomial $f(x,y) \in \FF_q[x,y]$ whose monomial terms all have degree at least two and all involve at least one factor of $x$.  Thus
\[\cF_Y(X) \overset{\mathrm{def}}= \cF(X+Y)\cdot \cF(Y)^{-1} = 1+(X +f(X,Y))\cdot \cF(Y)^{-1}= 1+X+\wt f(X,Y)\] where $\wt f(x,y) \in \FF_q[x,y]$ has the same properties as $f(x,y)$.  It follows that if $X \in \fk l_\lambda$ and $Y \in \fkn$ then $\wt f(X,Y) \in \ker(\lambda) \cap \fk l_\lambda$ and we have
$ \theta_\lambda\(\cF_Y(X)\) = \theta\circ \lambda(X)$.  Furthermore, if $Y \in \fkn$ is fixed then
$X \mapsto \cF_Y(X)$ defines a bijection $\fk l_\lambda \to L_\lambda$ since the map is clearly invertible.

Using the preceding observations with the fact that $\eta$ is a homomorphism, we now compute
\[\ba  \left\langle \theta_\lambda^\cF, \eta\right\rangle_G& = \frac{1}{|L_\lambda||G|} \sum_{X \in \fk l_\lambda} \sum_{Y \in \fkn} \theta\circ \lambda(X+Y) \cdot \overline{\eta\( \cF(X+Y)\)}
\\
&
= \frac{1}{|G|} \sum_{Y \in \fkn} \theta\circ\lambda(Y) \cdot \overline{\eta\(\cF(Y)\)}\( \frac{1}{|L_\lambda|} \sum_{X \in \fk l_\lambda} \theta_\lambda\(\cF_Y(X)\)  \cdot \overline{\eta\(\cF_Y(X)\)}\)
\\
&
= \frac{1}{|G|} \sum_{Y \in \fkn} \theta_\lambda(1+Y) \cdot \overline{\eta\(\cF(Y)\)} \cdot \left\langle \theta_\lambda, \Res_{L_\lambda}^G(\eta)\right\rangle_{L_\lambda}
\\
&
= \left\langle\theta_\lambda^\cF,\eta\right\rangle_{G} \cdot \left\langle\chi_\lambda,\eta\right\rangle_{G}.
\ea
\]  
We conclude from (\ref{thata}) that 
if $\eta \in \Irr(G)$ is linear then $\left\langle \chi_\lambda,\eta\right\rangle_G=0$ implies $\left\langle \psi_\lambda^\cF,\eta\right\rangle_G=0$.  
\end{proof}

We are now able to prove our theorem.  In the next two results, let $\cF : \fkn \to G$ be a polynomial map of the form (\ref{polynomial}).  Observe that this result is a more precise version of the third theorem promised in the introduction.

\begin{theorem}\label{otto-gen}
  If $\lambda \in \fkn^*$ and $\vartheta_1,\dots,\vartheta_m \in \Irr(G)$ are the irreducible constituents of the character $\xi_\lambda$, then there exist complex numbers $c_i$ such that $\psi_\lambda^\cF = \sum_{i=1}^m c_i \vartheta_i$.  
\end{theorem}

\begin{proof}  If  $\xi_\lambda=\chi_\lambda$ then one proceeds exactly as in the proof of \cite[Theorem 2.5]{O},  using the preceding lemma in place of \cite[Lemma 2.4]{O} and noting that $\Res_H^G(\psi^\cF) = \Res_H^G(\psi)^\cF$ for any algebra subgroup $H\subset G$.  The general case follows from Corollary \ref{andre-cor} using that $\Ind_H^G(\psi^\cF) = \Ind_H^G(\psi)^\cF$ along with part (1) in Proposition \ref{less}.  \end{proof}

While the first part of Proposition \ref{less} holds for the Kirillov functions $\psi_\lambda^\cF$ by elementary considerations, we are only now able to generalize the second part.

\begin{corollary}\label{less-cor}
If $\lambda \in \fkn^*$ and $\mu =\lambda \downarrow \olfks_\lambda$, then $\psi_\lambda^\cF \in \Irr(G)$ if and only if $\psi_\mu^\cF \in \Irr(\ols_\lambda)$.

\end{corollary}

\begin{proof}
Now that we know that $\psi_\mu^\cF$ is a linear combination of the characters in $\Irr(\ols_\lambda,\chi_\mu)$, we may deduce this result using the argument in the proof of Proposition \ref{less}. 
\end{proof}

Our next and final result could have appeared immediately after Proposition \ref{less}.  However, the statement of Corollary \ref{less-cor} helpfully diminishes the amount of discussion required in its proof.

\begin{proposition}\label{exp-criterion}  If $\lambda \in \fkn^*$ and $(\olfks_\lambda)^p \subset \olfkl_\lambda \cap \ker \lambda$ then $\psi_\lambda^\exp \in \Irr(G)$.
\end{proposition}

\begin{proof}
Let $\fk l = \olfkl_\lambda$ and $\fk s=\olfks_\lambda$.  
Write $\mu = \lambda \downarrow \fk s$ and recall that $\fk k_\mu\overset{\mathrm{def}} = \fk l \cap \ker \mu 
= \fk l \cap \ker \lambda$ is a two-sided ideal in $\fk s$.  
Let $\q = \fk s / \h$ be the quotient algebra;  by hypothesis, $\q^p= 0$ so every exponential Kirillov function of $Q = 1+\q$ is a character by Theorem \ref{Corollary 3}.   Obviously $\ker \mu \supset \fk k_\mu$ so it follows from Observation \ref{inflation} that $\logpsi_\mu$ is an irreducible character of $\ols_\lambda = 1+\fk s$, and hence  that $\logpsi_\lambda$ is an irreducible character of $G$ by the preceding corollary.
\end{proof}

\def\id{\mathrm{id}}


\begin{thebibliography}{99}


\bibitem{alia} M. Aguiar \emph{et al.}, 
``Supercharacters, symmetric functions in noncommuting variables, and related Hopf algebras,''  to appear in \emph{Adv. Math.}, {\tt{arXiv:1009.4134v1}} (2010).

\bibitem{Andre1} C. A. M. Andr\'e, ``Basic characters of the unitriangular group,'' \emph{J. Algebra} \textbf{175} (1995), 287--319.

\bibitem{AndreHecke} C. A. M. Andr\'e, ``Hecke algebras for the basic characters of the unitriangular groups,'' \emph{Proc. Amer.  Math. Soc.} \textbf{132} (2003), 987--996.

\bibitem{AndreAdjoint} C. A. M. Andr\'e; A. Nicol\'as, ``Supercharacters of the adjoint group of a finite radical ring,'' \emph{J. Group Theory} \textbf{11} (2008) 709--746. 





\bibitem{Boy} M. Boyarchenko, ``Base Change Maps for Unipotent Algebra Groups,'' preprint, {\tt{arXiv:math/0601133v1}} (2006).



\bibitem{DI} P. Diaconis; I. M. Isaacs, ``Supercharacters and superclasses for algebra groups,'' \emph{Trans. Amer. Math. Soc.} \textbf{360} (2008), 2359--2392. 






\bibitem{E} A. Evseev, ``Reduction for characters of finite algebra groups,'' \emph{J. Algebra} \textbf{325} (2010), 321--351.

\bibitem{H} Z. Halasi, ``On the characters and commutators of finite algebra groups,'' \emph{J. Algebra} \textbf{275} (2004), 481--487.




\bibitem{Isaacs} I. M. Isaacs, \emph{Character theory of finite groups}, Dover, New York, 1994.

\bibitem{I95} I. M. Isaacs, ``Characters of groups associated with finite algebras,'' \emph{J. Algebra} \textbf{177} (1995), 708--730.

\bibitem{IK1} I. M. Isaacs; D. Karagueuzian, ``Conjugacy in groups of upper triangular matrices,'' \emph{J. Algebra} \textbf{202} (1998), 704--711.

\bibitem{IK2} I. M. Isaacs; D. Karagueuzian, ``Erratum: Conjugacy in groups of upper triangular matrices,'' \emph{J. Algebra} \textbf{208} (1998), 722.

\bibitem{IK05} I. M. Isaacs; D. Karagueuzian,  ``Involutions and characters of upper triangular matrix groups,'' \emph{Mathematics of Computation} \textbf{252} (2005), 2027--2033. 




\bibitem{Kaz} D. Kazhdan, ``Proof of SpringerÕs hypothesis,'' \emph{Israel J. Math.} \textbf{28} (1977), 272--286.

\bibitem{K} A. A. Kirillov, ``Variations on the triangular theme,'' \emph{Lie groups and Lie algebras: E. B. DynkinÕs Seminar}, 43--73, Amer. Math. Soc. Transl. Ser. 2, \textbf{169} Providence, RI, 1995.














\bibitem{supp1} E. Marberg, ``Exotic characters of unitriangular matrix groups,'' \emph{J. Pure Appl. Algebra} \textbf{216} (2012), 239--254


\bibitem{supp2} E. Marberg, ``Combinatorial methods of character enumeration for the unitriangular group,'' \emph{J. Algebra} \textbf{345} (2011), 295--323

\bibitem{M1} E. Marberg, ``A supercharacter analogue for normality,'' \emph{J. Algebra} \textbf{332} (2011), 334--365. 


\bibitem{MT} E. Marberg; N. Thiem,  ``Superinduction for pattern groups,'' \emph{J. Algebra} \textbf{321} (2009), 3681--3703. 

\bibitem{O} B. Otto,  ``Constituents of supercharacters and Kirillov functions,'' \emph{Archiv der Mathematik}, \textbf{94} (2010), 319--326.

\bibitem{Sangroniz} J. Sangroniz, ``Characters of algebra groups and unitriangular groups,'' \emph{Finite groups 2003}, 335--349, Walter de Gruyter, Berlin, 2004.








\bibitem{VeraLopez2004} A. Vera-Lopez; J. M. Arregi, ``Computing in unitriangular matrices over finite fields,'' \emph{J. Linear Algebra Appl.} \textbf{387} (2004), 193--219.

\bibitem{Yan} N. Yan, \emph{Representation theory of the finite unipotent linear groups,} PhD thesis, Department of Mathematics, University of Pennsylvania, 2001.

\end{thebibliography}
\end{document}